\documentclass[a4paper,12pt,twoside]{amsart}
\usepackage[english]{babel}
\usepackage{amssymb, amsmath, eucal}
\usepackage[all]{xy}
\usepackage{mathrsfs}
\usepackage[dvips]{graphics}

\newtheorem{The}{Theorem}[section]

\newtheorem{Prop}[The]{Proposition}

\newtheorem{Question}[The]{Question}

\newtheorem{Ex}[The]{Example}

\newtheorem*{ackn}{Acknowledgements}

\newcommand{\C}{\mathbb{C}}
\newcommand{\R}{\mathbb{R}}
\newcommand{\N}{\mathbb{N}}

\begin{document}
 \title[ Maximal plurisubharmonic function]{Some properties of Maximal plurisubharmonic function}

\setcounter{tocdepth}{1}

  \author{Do Hoang Son } 
\address{Institute of Mathematics \\ Vietnam Academy of Science and Technology \\18
Hoang Quoc Viet \\Hanoi \\Vietnam}
\email{hoangson.do.vn@gmail.com }
\email{dhson@math.ac.vn}
 \date{\today}


\begin{abstract}
  The purpose of this paper is to provide some properties of maximal plurisubharmonic functions in bounded domains in $\C^n$. 
\end{abstract}

\maketitle

\tableofcontents
\newpage

\section*{Introduction}
 Let $\Omega\subset\C^n$ be a bounded domain $(n\geq 2)$. A function $u\in PSH(\Omega)$ is called maximal
 if for every open set $G\Subset\Omega$, and for each upper semicontinuous function $v$ on $\bar{G}$ such that
 $v\in PSH(G)$ and $v|_{\partial G}\leq u|_{\partial G}$, we have $v\leq u$. There are some equivalent descriptions of
 maximality which have been presented in \cite{Sad81}, \cite{Kli91}.
 
 For the convenience, we denote\\
 \begin{flushleft}
 	$MPSH(\Omega)=\{u\in PSH(\Omega)| u \mbox{ is maximal }\}$,\\[12pt]
 	$MPSH_{loc}(\Omega)=\{u\in PSH(\Omega)| \forall z\in\Omega: u\in MPSH(U) \mbox{ for some } z\in U\Subset\Omega \}$.\\[12pt]
 \end{flushleft}
 
 The following question was given by Blocki  \cite{Blo04}, \cite{DGZ16}:
 \begin{Question}\label{bkquest}
 	Is $MPSH (\Omega)$ equal to $MPSH_{loc}(\Omega)$? (or is maximality a local notion?)
 \end{Question}
  Denote by $D(\Omega)$ the domain of definition of Monge-Amp\`ere operator in $\Omega$. By \cite{Blo04},
   a function $u\in D(\Omega)$ is maximal iff $(dd^cu)^n=0$. Moreover, it follows from  \cite{Blo06} that to belong
   to the class $D$ is a local property. Then 
   \begin{center}
   	$MPSH(\Omega)\cap D(\Omega)=MPSH_{loc}(\Omega)\cap D(\Omega)$.
   \end{center}
   In the general case, the question \ref{bkquest} is still open. It raises another question
   \begin{Question}
   	What information can be obtained from the condition $u\in MPSH_{loc}(\Omega)$?
   \end{Question}
   In this paper, we study on some subclass of $MPSH(\Omega)$ and use results on it to show some properties of 
   the class $MPSH_{loc}(\Omega)$.

  Our main results  are following:
  \begin{The}\label{app1}
  	If $u,v\in MPSH_{loc}(\Omega)$ and $\chi: \R\rightarrow\R$ is a convex non-decreasing function then
  	$(z,w)\mapsto\chi (u(z)+v(w))\in MPSH(\Omega\times\Omega)$.
  \end{The}
 \begin{The}\label{app2}
 	Let $u$ be a negative maximal plurisubharmonic function in $\Omega$ and let $U, \tilde{U}$ be open subset of
 	$\Omega$ such that $U\Subset\tilde{U}\Subset\Omega$. Assume that $u_j\in PSH^-(\tilde{U})\cap C(\tilde{U})$ is
 	decreasing to $u$ in $\tilde{U}$. Then
 	\begin{equation}\label{app2.eq}
 	\int\limits_U|u_j|^{-a}(dd^cu_j)^n\stackrel{j\to\infty}{\longrightarrow}0,\,\forall a>n-1.
 	\end{equation}
 \end{The}
  
  \begin{ackn}
  	I am thankful to Nguyen Quang Dieu and Nguyen Xuan Hong for useful discussion and comments.
  \end{ackn}
   
   \section{A class of maximal plurisubharmonic functions}
    We say that a function $u\in PSH^{-}(\Omega)$ has $M_1$ property iff for 
    every open set $U\Subset\Omega$, there are $u_j\in PSH^{-}(U)\cap C(U)$ such that $u_j$ is decreasing to $u$ in $U$ and
    \begin{equation}\label{M1prop.eq}
    \lim\limits_{j\to\infty}\left(\int\limits_{U\cap\{u_j>-t\}}(dd^c u_j)^n
    +\int\limits_{U\cap\{u_j>-t\}}du_j\wedge d^cu_j\wedge(dd^cu_j)^{n-1}\right)=0,
    \end{equation}
    for any $t>0$. We denote by $M_1PSH(\Omega)$ the set of negative plurisubharmonic functions in $\Omega$ satisfying $M_1$ property.\\
    If  $\chi: \R\rightarrow\R$ is a convex non-decreasing function, we denote $MPSH_{\chi}(\Omega)$ the set of 
    negative plurisubharmonic functions in $\Omega$ such that $\chi(u)\in MPSH(\Omega)$.\\
    
    \begin{The}\label{M1.the}
    	Let $\Omega$ be a bounded domain in $\C^n$ and $u\in PSH^{-}(\Omega)$. Then the following conditions are equivalent
    	\begin{itemize}
    		\item [(i)] $u\in M_1PSH(\Omega)$.
    		\item [(ii)] $u\in MPSH_{\chi}(\Omega)$ for any  convex non-decreasing function $\chi: \R\rightarrow\R$.
    		\item [(iii)] For any open sets $U,\tilde{U}$ such that $U\Subset\tilde{U}\Subset\Omega $, for any 
    		$u_j\in PSH^{-}(\tilde{U})\cap C(\tilde{U})$ such that $u_j$ is decreasing to $u$ in $\tilde{U}$, we have
    		\begin{center}
    			$\lim\limits_{j\to\infty}\left(\int\limits_U|u_j|^{-a}(dd^cu_j)^n
    			+\int\limits_U|u_j|^{-a-1}du_j\wedge d^cu_j\wedge (dd^cu_j)^{n-1}\right)=0,$
    		\end{center}
    		for all $a>n-1$.
    	\end{itemize}
    	In particular, $M_1$ property is a local notion and $M_1PSH(\Omega)\subset MPSH(\Omega)$.
    \end{The}
    \begin{proof}
    	$(iii\Rightarrow i)$: Obvious.\\
    	$(i\Rightarrow ii)$:\\
    	Assume that $U\Subset\tilde{U}\Subset\Omega$. Let
    	$u_j\in PSH^{-}(U)\cap C(U)$ such that $u_j$ is decreasing to $u$ in $U$ and the condition \eqref{M1prop.eq} is satisfied.
    	
    	If $\chi$ is smooth and $\chi$ is constant in some interval $(-\infty, -m)$ then
    	\begin{flushleft}
    		$\begin{array}{ll}
    		(dd^c\chi(u_j))^n &=(\chi '(u_j))^n(dd^cu_j)^n+n\chi ''(u_j)(\chi '(u_j))^{n-1}du_j\wedge d^cu_j\wedge (dd^cu_j)^{n-1}\\[12pt]
    		&\leq C \boldsymbol{1}_{\{u_j>-t\}}(dd^cu_j)^n+C\boldsymbol{1}_{\{u_j>-t\}}du_j\wedge d^cu_j\wedge (dd^cu_j)^{n-1},
    		\end{array}$
    	\end{flushleft}
    	where $C,t>0$ depend only on $\chi$. Hence
    	\begin{center}
    		$\int\limits_U	(dd^c\chi(u_j))^n\stackrel{j\to\infty}{\longrightarrow}0.$
    	\end{center}
    	Then, by \cite{Sad81} (see also \cite{Ceg09}), $\chi (u)$ is maximal on $U$ for any open set $U\Subset\Omega$. Thus
    	$\chi (u)\in MPSH(\Omega)$.\\
    	In the general case, for any convex non-decreasing function $\chi$, we can find $\chi_l\searrow\chi$ such that $\chi_l$
    	is smooth, convex and $\chi|_{(-\infty, -m)}=const$ for some $m$.
    	 By above argument, $\chi_l\in MPSH(\Omega)$ for any $l\in\N$.
    	Hence $\chi (u)\in MPSH(\Omega).$\\
    	$(ii\Rightarrow iii)$:\\
    	For any $0<\alpha<\frac{1}{n}$, the function
    	\begin{center}
    		$\Phi_{\alpha}(t)=-(-t)^{\alpha}$
    	\end{center}
    	is convex and non-decreasing in $\R^-$. Assume that $u$ satisfies (ii), we have $\Phi_{\alpha}\in MPSH(\Omega)$.\\
    	
    	By \cite{Bed93} (see also \cite{Blo09}), for any $0<\alpha<\frac{1}{n}$, we have $\Phi_{\alpha}(u)\in D(\Omega)$. Then,
    	for any  $u_j\in PSH^{-}(\tilde{U})\cap C(\tilde{U})$ such that $u_j$ is decreasing to $u$ in $\tilde{U}$, we have
    	\begin{center}
    		$\int\limits_U(dd^c\Phi_{\alpha}(u_j))^n\stackrel{j\to\infty}{\longrightarrow}0, \forall 0<\alpha<\frac{1}{n},$
    	\end{center}
    	and it implies (iii).
    	
    	Finally, by using $(i\Leftrightarrow iii)$, we conclude that $M_1$ property is a local notion.
    \end{proof}
    The following proposition is an immediately corollary of Theorem \ref{M1.the}
    \begin{Prop}
    	Let $\Omega$ be a bounded domain in $\C^n$.
    	\begin{itemize}
    		\item [(i)] If $u\in M_1PSH(\Omega)$ then $\chi (u)\in M_1PSH(\Omega)$ 
    		for any  convex non-decreasing function $\chi: \R^-\rightarrow\R^-$.
    		\item [(ii)] If $u_j\in M_1PSH(\Omega)$ and $u_j$ is decreasing to $u$ then $u\in M_1PSH(\Omega)$.
    		\item[(iii)] Let $u\in PSH^-(\Omega)\cap C^2(\Omega\setminus F)$, where $F=\{z: u(z)=-\infty\}$ is closed. If
    		\begin{center}
    			$(dd^c u)^n=du\wedge d^cu\wedge (dd^c u)^{n-1}=0$
    		\end{center}
    		in $\Omega\setminus F$ then $u\in M_1PSH(\Omega)$.
    	\end{itemize}
    \end{Prop}
    In some special cases, we can easily check $M_1$ property by the following criteria
    \begin{Prop}\label{M1crit.prop}
    	Let $\Omega$ be a bounded domain in $\C^n$. Let $\chi:\R\rightarrow\R$ be a smooth convex increasing function such that
    	$\chi'' (t)>0$ for any $t\in\R$. Assume also that $\chi$ is lower bounded.
    	 If $u\in PSH^-(\Omega)$ and $\chi (u)\in MPSH(\Omega)$ then $u\in M_1PSH(\Omega)$.
    \end{Prop}
    \begin{proof}
    	Let $U\Subset\tilde{U}\Subset\Omega$ and $u_j\in PSH(\tilde{U})\cap C(\tilde{U})$ such that $u_j$ is decreasing to $u$. Then
    	\begin{center}
    		$dd^c(\chi(u_j))=\chi'(u_j)dd^cu_j+\chi''(u_j)du_j\wedge d^cu_j$
    	\end{center}
    	and
    	\begin{center}
    		$	(dd^c\chi(u_j))^n=(\chi '(u_j))^n(dd^cu_j)^n+n\chi''(u_j)(\chi'(u_j))^{n-1} du_j\wedge d^cu_j\wedge (dd^c u_j)^{n-1}.$
    	\end{center}
    For any $t>0$, there exists $C>0$ depending only on $t$ and $\chi$ such that
    \begin{equation}\label{M1crit.eq1}
    	(dd^c\chi(u_j))^n
    	\geq C \boldsymbol{1}_{\{u_j>-t\}}(dd^cu_j)^n+C \boldsymbol{1}_{\{u_j>-t\}} du_j\wedge d^cu_j\wedge (dd^c u_j)^{n-1}.
    \end{equation}
    Note that $\chi (u)\in D(\Omega)\cap MPSH(\Omega)$. Hence
    \begin{equation}\label{M1crit.eq2}
    \lim\limits_{j\to\infty}\int\limits_U(dd^c\chi(u_j))^n=0.
    \end{equation}	
    Combining \eqref{M1crit.eq1} and \eqref{M1crit.eq2}, we have
    \begin{center}
    	$\lim\limits_{j\to\infty}\left(\int\limits_{U\cap\{u_j>-t\}}(dd^c u_j)^n
    	+\int\limits_{U\cap\{u_j>-t\}}du_j\wedge d^cu_j\wedge(dd^cu_j)^{n-1}\right)=0.$
    \end{center}
    Thus $u\in M_1PSH(\Omega)$.
    \end{proof}
    \begin{Ex}
    	(i) If $u$ is a negative plurisubharmonic function in $\Omega\subset\C^n$ depending only on $n-1$ variables then $u$ has $M_1$
    	property.\\
    	(ii) If $f: \Omega\rightarrow \C^n$ is a holomorphic mapping of rank $<n$ then $(dd^c|f|^2)^n=0$
    	 (see, for example, in \cite{Ras98}). Then, by Proposition \ref{M1crit.prop}, $\log |f|\in M_1PSH(\Omega)$
    	 if it is negative in $\Omega$.
    \end{Ex}
    \begin{Question}
    	Does Proposition \ref{M1crit.prop} still hold if the assumption ``$\chi$ is lower bounded'' is removed from it?
    \end{Question}
   \section{Proof of the main theorems}
   \subsection{Proof of Theorem \ref{app1}}
   	Without loss of generality, we can assume that $u,v\in PSH^-(\Omega)$.
   	
   	If $u,v\in MPSH_{loc}(\Omega)$ then for any $z_0, w_0\in\Omega$, there are hyperconvex domains $U, \tilde{U}, V, \tilde{V}$ such that
   	$z_0\in U\Subset\tilde{U}\Subset\Omega$, $w_0\in V\Subset\tilde{V}\Subset\Omega$, $u\in MPSH(\tilde{U})$
   	and $v\in MPSH(\tilde{V})$. We need to show that $u(z)+v(w)$ have $M_1$ property in $U\times V$.
   	
   	Let $u_j\in PSH^-(\tilde{U})\cap C(\tilde{U})$ and $v_j\in PSH^-(\tilde{V})\cap C(\tilde{V})$ such that
   	$u_j$ is decreasing to $u$ in $\tilde{U}$ and $v_j$ is decreasing to $v$ in $\tilde{V}$. By \cite{Wal68},
   	there are $\tilde{u}_j\in PSH^-(\tilde{U})\cap C(\tilde{U})$, $\tilde{v}_j\in PSH^-(\tilde{V})\cap C(\tilde{V})$
   	such that
   	\begin{center}
   		$\begin{cases}
   		\tilde{u_j}=u_j\quad\mbox{ in }\tilde{U}\setminus U,\\
   		\tilde{v_j}=v_j\quad\mbox{ in }\tilde{V}\setminus V,\\
   		(dd^c\tilde{u}_j)^n=0\quad\mbox{ in }U,\\
   		(dd^c\tilde{v}_j)^n=0\quad\mbox{ in }V.
   		\end{cases}$
   	\end{center}
   	By the maximality of $u$ and $v$, we conclude that $\tilde{u}_j$ is decreasing to $u$ in $\tilde{U}$ and
   	$\tilde{v}_j$ is decreasing to $v$ in $\tilde{V}$. In $U\times V$, we have
   	\begin{flushleft}
   		$(dd^c(\tilde{u}_j(z)+\tilde{v}_j(w)))^{2n}=C^n_{2n}(dd^c\tilde{u}_j)^n_z\wedge (dd^c \tilde{v}_j)^n_w=0$\\[12pt]
   		$d(\tilde{u}_j(z)+\tilde{v}_j(w))\wedge d^c(\tilde{u}_j(z)+\tilde{v}_j(w))\wedge (dd^c(\tilde{u}_j(z)+\tilde{v}_j(w)))^{2n-1}$\\[12pt]
   		$= C^{n-1}_{2n-1}d_z\tilde{u}_j\wedge d^c_z\tilde{u}_j\wedge (dd^c \tilde{u}_j)^{n-1}_z\wedge (dd^c  \tilde{v}_j)^n_w$
   		$+C^{n-1}_{2n-1}d_w \tilde{v}_j\wedge d^c_w \tilde{v}_j\wedge (dd^c  \tilde{v}_j)^{n-1}_w\wedge (dd^c \tilde{u}_j)^n_z$\\[12pt]
   		$=0.$
   	\end{flushleft}
   	
   	Then $u(z)+v(w)$ has $M_1$ property in $U\times V$. By Theorem \ref{M1.the}, $M_1$ property is a local notion. Hence
   	$u(z)+v(w)\in M_1PSH(\Omega\times\Omega)$. And it implies that $u(z)+v(w)\in MPSH_{\chi}(\Omega\times\Omega)$
   	for any convex non-decreasing function $\chi$.
   \subsection{Proof of Theorem \ref{app2}}
   	Let $v=|z_1|^2+...+|z_{n-1}|^2+x_n+y_n-M$, where $M=\sup\limits_{\Omega}(|z|^2+|x_n|+|y_n|)$. Then $v\in MPSH(\Omega)$.
   	By Theorem \ref{app1}, $\chi(u(z)+v(w))\in MPSH(\Omega\times\Omega)$ for any convex non-decreasing function $\chi$.
   	
   	By \cite{Bed93},\cite{Blo09}, for any $0<\alpha<\frac{1}{2n}$, we have $\Phi_{\alpha}(u(z)+v(w))\in D(\Omega\times\Omega)$,
   	where $\Phi_{\alpha}$ is defined as in the proof of Theorem \ref{M1.the}.\\
   	Then
   	\begin{center}
   		$\int\limits_{U\times U}(dd^c\Phi(u_j(z)+v(w)))^{2n}\stackrel{j\to\infty}{\longrightarrow}0,$
   	\end{center}
   	for any $0<\alpha<\frac{1}{2n}$. Hence
   	\begin{equation}\label{1proofapp2.eq}
   	\int\limits_U|u_j|^{-2n-1+2n\alpha}(dd^cu_j)^n\stackrel{j\to\infty}{\longrightarrow}0,\,\forall 0<\alpha<\frac{1}{2n}.
   	\end{equation}
   	Moreover, $\Phi_{\beta}(u)\in D(\Omega)$ for any $0<\beta<\frac{1}{n}$. Then, for any $0<\beta<\frac{1}{n}$,
   	there is $C_{\beta}>0$ such that
   	\begin{center}
   		$\int\limits_U(dd^c\Phi_{\beta}(u_j))^n\leq C_{\beta},\,\forall j>0.$
   	\end{center}
   	Hence
   	\begin{equation}\label{2proofapp2.eq}
   	\int\limits_U|u_j|^{-n+n\beta}(dd^cu_j)^n\leq C_{\beta},\,\forall j>0, \forall 0<\beta<\frac{1}{n}.
   	\end{equation}
   	Combining \eqref{1proofapp2.eq}, \eqref{2proofapp2.eq} and using H\"older inequality, we obtain \eqref{app2.eq}.
   \section{Relation between some class of maximal plurisubharmonic functions}
   Let $\Omega$ be a bounded domain in $\C^n$. Let $u\in PSH(\Omega)$. If there exists a sequence of convex
   non-decreasing functions 
   $\chi^m:\R\rightarrow\R$ such that
   \begin{itemize}
   	\item $\chi^m$ is lower bounded for every $m$,
   	\item $\chi^m$ is decreasing to $Id$ as $m\rightarrow\infty$,
   	\item $(dd^c\chi^m(u))^n\stackrel{m\to\infty}{\longrightarrow}0$ in the weak sense, 
   \end{itemize}
   then, by \cite{Sad81}, $u$ is maximal. We are interested in the following question
   \begin{Question}\label{appro.quest}
   	If $u$ is maximal, does there exist a sequence of convex non-decreasing function $\chi^m$ satisfying above conditions?
   	If it exists, how to find it?
   \end{Question}
   In this section we discuss about relation between some class of maximal plurisubharmonic functions in a bounded domain
   $\Omega$ in $\C^n$. It can be seen as the first step in approaching Question \ref{appro.quest}.
   
    Assume that $\chi: \R\rightarrow\R$ is a smooth convex function such that
   $\chi|_{(-\infty, -2)}=-1$, $\chi|_{(0,\infty)}=Id_{(0,\infty)}$ and $\chi''(-1)>0$. We denote
   \begin{flushleft}
   	$M_2PSH(\Omega)=\{u\in PSH^-(\Omega)| (dd^c\max\{u, -k\})^n\stackrel{weak}{\longrightarrow}0\mbox{ as }k\rightarrow\infty\}$,\\
   	$M_3PSH(\Omega)=\{u\in PSH^-(\Omega)| (dd^c\chi_k(u))^n\stackrel{weak}{\longrightarrow}0\mbox{ as }k\rightarrow\infty\}$,\\
   	$M_4PSH(\Omega)=\{u\in PSH^-(\Omega)| (dd^c\log(e^u+\dfrac{1}{k}))^n 
   	\stackrel{weak}{\longrightarrow}0\mbox{ as }k\rightarrow\infty\},$
   \end{flushleft}
   where $\chi_k(t)=\chi (t+k)-k$. The main result of this section is following
   \begin{The}\label{rel.the}
   	$M_2PSH(\Omega)\subset M_3PSH(\Omega)$ and $M_3PSH(\Omega)=M_4PSH(\Omega).$
   \end{The}
   First, we introduce some characteristics of $M_2PSH(\Omega), M_3PSH(\Omega)$ and $M_4PSH(\Omega)$.
   \begin{Prop}\label{M2.prop}
   	Let $u\in PSH^{-}(\Omega)$. Assume that $u_j\in PSH^{-}(\Omega)\cap C(\Omega)$ is decreasing to $u$. Then the following
   	conditions are equivalent
   	\begin{itemize}
   		\item [(i)] $u\in M_2PSH(\Omega)$.
   		\item [(ii)] If $U\Subset\Omega$ then for every $M>0$
   		\begin{equation}\label{M2.eq1}
   		\int\limits_{U\cap\{u_j>-M\}}(dd^c u_j)^n\stackrel{j\to\infty}{\longrightarrow}0,
   		\end{equation}
   		and for any $\epsilon>0$, there is $k_0>0$ such that
   		\begin{equation}\label{M2.eq2}
   		\forall k\geq k_0,\exists l_0>0: \limsup\limits_{j\to\infty}
   		\int\limits_{U\cap\{-k-\frac{1}{l}<u_j<-k+\frac{1}{l}\}}du_j\wedge d^cu_j\wedge (dd^cu_j)^{n-1}
   		<\dfrac{\epsilon}{l}, \forall l>l_0.
   		\end{equation}
   	\end{itemize}
   \end{Prop}
   \begin{proof}
   	For any $k,l>0$, we denote
   	\begin{center}
   		$\chi_{k,l}(t)=\dfrac{1}{l}\chi ((t+k-\dfrac{1}{l})l)-k+\dfrac{1}{l}.$
   	\end{center}
   	Then $\chi_{k,l}(u)$ converges to $\max\{u,-k\}$ as $l\rightarrow\infty$. Moreover
   	\begin{center}
   		$\chi_{k,l}(u)\geq\max\{u,-k\}+O(\dfrac{1}{l}).$
   	\end{center}
   	Hence, by \cite{Ceg09},
   	\begin{equation}\label{M2proof.eq1}
   	(dd^c\chi_{k,l}(u))^n\stackrel{l\to\infty}{\longrightarrow}(dd^c\max\{u,-k\})^n
   	\end{equation}
   	in the weak sense.
   	
   	For any $j>0$, we have
   	\begin{center}
   		$(dd^c\chi_{k,l}(u_j))^n=(\chi'((u_j+k-\dfrac{1}{l})l))^n(dd^cu_j)^n
   		+nl\chi''((u_j+k-\dfrac{1}{l})l)(\chi'((u_j+k-\dfrac{1}{l})l))^{n-1}du_j\wedge d^cu_j\wedge (dd^cu_j)^{n-1}.$
   	\end{center}
   	Then, there are $C_1,C_2,\delta>0$ depending only on $\chi$ such that
   	\begin{center}
   		$(dd^c\chi_{k,l}(u_j))^n\leq
   		C_1\boldsymbol{1}_{\{u_j>-k-\frac{1}{l}\}}(dd^cu_j)^n+
   		 C_1l\boldsymbol{1}_{\{-k-\frac{1}{l}<u_j<-k+\frac{1}{l}\}}du_j\wedge d^cu_j\wedge (dd^cu_j)^{n-1},$
   	\end{center}
   	and
   	\begin{center}
   			$(dd^c\chi_{k,l}(u_j))^n\geq
   			C_2\boldsymbol{1}_{\{u_j>-k+\frac{1}{l}\}}(dd^cu_j)^n+
   			C_2l\boldsymbol{1}_{\{-k-\frac{\delta}{l}<u_j<-k+\frac{\delta}{l}\}}du_j\wedge d^cu_j\wedge (dd^cu_j)^{n-1}.$
   	\end{center}
   	Moreover,
   	\begin{center}
   			$(dd^c\chi_{k,l}(u_j))^n\stackrel{j\to\infty}{\longrightarrow}(dd^c\chi_{k,l}(u))^n$
   	\end{center}
   	in the weak sense.
   	
   	Then $(ii)$ is equivalent to
   	\begin{center}
   		$\limsup\limits_{l\to\infty}\int\limits_U(dd^c\chi_{k,l}(u))^n\stackrel{k\to\infty}{\longrightarrow} 0,$
   	\end{center}
   	for any $U\Subset\Omega$.
   	
   Hence, by \eqref{M2proof.eq1}, we conclude that $(i)$ is equivalent to $(ii)$.
   	\end{proof}
   \begin{Prop}\label{M3.prop}
   	Let $u\in PSH^{-}(\Omega)$. Assume that $u_j\in PSH^{-}(\Omega)\cap C(\Omega)$ is decreasing to $u$. Then the following
   	conditions are equivalent
   	\begin{itemize}
   	\item [(i)] $u\in M_3PSH(\Omega)$.
   	\item [(ii)] If $U\Subset\Omega$ then
   	\begin{equation}\label{M3.eq1}
   	\int\limits_{U\cap\{u_j>-M\}}(dd^c u_j)^n\stackrel{j\to\infty}{\longrightarrow}0,
   	\end{equation}
   	and 
   	\begin{equation}\label{M3.eq2}
   	\limsup\limits_{j\to\infty}\int\limits_{U\cap\{-k-1<u_j<-k\}}du_j\wedge d^cu_j\wedge (dd^cu_j)^{n-1}
   	\stackrel{k\to\infty}{\longrightarrow}0.
   	\end{equation}
   \end{itemize}
   \end{Prop}
   The proof of Proposition \ref{M3.prop} is similar to the proof of Proposition \ref{M2.prop}. We leave the details to the reader.
   
   In \cite{Sad12}, Sadullaev has proved a characteristic of $M_4PSH(\Omega)$:
   \begin{The}\label{Sad.the}
   	Let $u\in PSH^-(\Omega)$. Denote $v=e^u$. Then $u\in M_4PSH(\Omega)$ iff
   	\begin{center}
   		$v(dd^c v)^n-ndv\wedge d^cv\wedge (dd^cv)^{n-1}=0,$
   	\end{center}
   	and
   	\begin{center}
   		$\lim\limits_{t\to 0}\dfrac{1}{t^n}\int\limits_{U\cap\{v<t\}}(dd^cv)^n=0,$
   	\end{center}
   	for any $U\Subset\Omega$.
   \end{The}
  Using Theorem \ref{Sad.the}, it is easy to show that
   \begin{Prop}\label{M4.prop}
   		Let $u\in PSH^{-}(\Omega)$. Assume that $u_j\in PSH^{-}(\Omega)\cap C(\Omega)$ is decreasing to $u$. Then the following
   		conditions are equivalent
   		\begin{itemize}
   			\item [(i)] $u\in M_4PSH(\Omega)$.
   			\item [(ii)] If $U\Subset\Omega$ then
   			\begin{equation}\label{M4.eq1}
   			\int\limits_{U\cap\{u_j>-M\}}(dd^c u_j)^n\stackrel{j\to\infty}{\longrightarrow}0,
   			\end{equation}
   			and 
   			\begin{equation}\label{M4.eq2}
   			\limsup\limits_{j\to\infty}\int\limits_{U\cap\{u_j<-k\}}e^{n(u_j+k)}
   			du_j\wedge d^cu_j\wedge (dd^cu_j)^{n-1}
   			\stackrel{k\to\infty}{\longrightarrow}0.
   			\end{equation}
   		\end{itemize}
   	\end{Prop}
   	Now, by using Propositions \ref{M2.prop}, \ref{M3.prop} and \ref{M4.prop}, we will prove Theorem \ref{rel.the}.
   	\begin{proof}[Proof of Theorem \ref{rel.the}]
   		Let $u\in PSH^-(\Omega)$. By replacing $\Omega$ by an exhaustive sequence of relative compact subsets of $\Omega$,
   we can assume that there exists a sequence  $u_j\in PSH^{-}(\Omega)\cap C(\Omega)$ such that $u_j$ decreasing to $u$
   in $\Omega$.
   
   	Assume that $u\in M_2PSH(\Omega)$. By Proposition \ref{M2.prop}, for any $\epsilon>0$, there exist $k_0>0$ such that
   	\begin{center}
   		$\forall k\geq k_0,\exists l_0>0: \limsup\limits_{j\to\infty}
   		\int\limits_{U\cap\{-k-\frac{1}{l}<u_j<-k+\frac{1}{l}\}}du_j\wedge d^cu_j\wedge (dd^cu_j)^{n-1}
   		<\dfrac{\epsilon}{l}, \forall l>l_0.$
   	\end{center}
   	Let $k>k_0$. By Besicovitch's covering theorem \cite{Matt95}, there are $k_1, l_1,...,k_m,l_m...>0$ such that
   	\begin{center}
   		$\limsup\limits_{j\to\infty}
   		\int\limits_{U\cap\{-k_m-\frac{1}{l_m}<u_j<-k_m+\frac{1}{l_m}\}}du_j\wedge d^cu_j\wedge (dd^cu_j)^{n-1}
   		<\dfrac{\epsilon}{l_m}, \forall m>0,$
   	\end{center}
   	and
   	\begin{center}
   		$\boldsymbol{1}_{(-k-1, -k)}\leq\sum\limits_m\boldsymbol{1}_{(-k_m-\frac{1}{l_m}, -k_m+\frac{1}{l_m})}
   		=\sum\limits_m\dfrac{2}{l_m}\leq C,$
   	\end{center}
   	where $C>0$ is a universal constant. Hence
   	\begin{center}
   		$\limsup\limits_{j\to\infty}
   		\int\limits_{U\cap\{-k-1<u_j<-k\}}du_j\wedge d^cu_j\wedge (dd^cu_j)^{n-1}
   		\leq C\epsilon.$
   	\end{center}
   	By Proposition \ref{M3.prop}, we get $u\in M_3PSH(\Omega)$.
   	
   	Thus $M_2PSH(\Omega)\subset M_3PSH(\Omega)$.
   	
   	Now, assume that $u$ is an arbitrary element of $M_3PSH(\Omega)$. By Proposition \ref{M3.prop}, for any $\epsilon>0$, 
   	there exists $k_0>0$ such that
   	\begin{center}
   		$\limsup\limits_{j\to\infty}\int\limits_{U\cap\{-k-1<u_j<-k\}}du_j\wedge d^cu_j\wedge (dd^cu_j)^{n-1}
   		<\epsilon.$
   	\end{center}
   	for every $k>k_0$. Then
   	\begin{flushleft}
   		$\limsup\limits_{j\to\infty}\int\limits_{U\cap\{u_j<-k\}}e^{n(u_j+k)}
   		du_j\wedge d^cu_j\wedge (dd^cu_j)^{n-1}$\\
   		$\leq\sum\limits_{l=0}^{\infty}\limsup\limits_{j\to\infty}\int\limits_{U\cap\{-k-l-1<u_j<-k-l\}}
   		e^{n(u_j+k)}
   		du_j\wedge d^cu_j\wedge (dd^cu_j)^{n-1}$\\
   		$\leq\sum\limits_{l=0}^{\infty}\limsup\limits_{j\to\infty}\int\limits_{U\cap\{-k-l-1<u_j<-k-l\}}
   		e^{-nl}
   		du_j\wedge d^cu_j\wedge (dd^cu_j)^{n-1}$\\
   		$\leq\epsilon\sum\limits_{l=0}^{\infty}e^{-nl},$
   	\end{flushleft}
   	for any $k>k_0$. Hence 
   	\begin{center}
   			$\limsup\limits_{j\to\infty}\int\limits_{U\cap\{u_j<-k\}}e^{n(u_j+k)}
   			du_j\wedge d^cu_j\wedge (dd^cu_j)^{n-1}
   			\stackrel{k\to\infty}{\longrightarrow}0.$
   	\end{center}
   	By Proposition \ref{M4.prop}, we get $u\in M_4PSH(\Omega)$.
   	
   	Thus $M_3PSH(\Omega)\subset M_4PSH(\Omega)$.
   	
   	Conversely, if $u\in M_4PSH(\Omega)$ then
   	\begin{flushleft}
   		$\limsup\limits_{j\to\infty}\int\limits_{U\cap\{-k-1<u_j<-k\}}du_j\wedge d^cu_j\wedge (dd^cu_j)^{n-1}$\\
   		$\leq e^n\limsup\limits_{j\to\infty}\int\limits_{U\cap\{-k-1<u_j<-k\}}e^{n(u_j+k)}
   		du_j\wedge d^cu_j\wedge (dd^cu_j)^{n-1}$\\
   		$\leq e^n\limsup\limits_{j\to\infty}\int\limits_{U\cap\{u_j<-k\}}e^{n(u_j+k)}
   		du_j\wedge d^cu_j\wedge (dd^cu_j)^{n-1}$\\
   		$\stackrel{k\to\infty}{\longrightarrow} 0.$
   	\end{flushleft}
   	
   	Hence $M_4PSH(\Omega)\subset M_3PSH(\Omega)$.
   	
   	The proof is completed.
   	\end{proof}

\end{document}